\documentclass[11pt,oneside,english]{amsart}
\usepackage[T1]{fontenc}
\usepackage[latin9]{inputenc}
\usepackage{geometry}
\usepackage{babel}
\usepackage{amstext}
\usepackage{amsthm}
\usepackage{amssymb}
\usepackage[unicode=true,
 bookmarks=true,bookmarksnumbered=false,bookmarksopen=false,
 breaklinks=false,pdfborder={0 0 1},backref=false,colorlinks=false]
 {hyperref}

\makeatletter
\numberwithin{equation}{section}
\numberwithin{figure}{section}
\theoremstyle{plain}
\newtheorem{thm}{\protect\theoremname}[section]
\theoremstyle{plain}
\newtheorem{lem}[thm]{\protect\lemmaname}
\theoremstyle{plain}
\newtheorem{prop}[thm]{\protect\propositionname}
\theoremstyle{plain}
\newtheorem{cor}[thm]{\protect\corollaryname}
\theoremstyle{remark}
\newtheorem*{rem*}{\protect\remarkname}

\IfFileExists{lmodern.sty}{\usepackage{lmodern}}{}

\date{\today}

\makeatother

\providecommand{\corollaryname}{Corollary}
\providecommand{\lemmaname}{Lemma}
\providecommand{\propositionname}{Proposition}
\providecommand{\remarkname}{Remark}
\providecommand{\theoremname}{Theorem}

\begin{document}
\title{Intermediate-scale statistics for real-valued lacunary sequences}
\author{Nadav Yesha}
\address{Department of Mathematics, University of Haifa, 3498838 Haifa, Israel.}
\email{nyesha@univ.haifa.ac.il}
\thanks{We thank Jens Marklof and Zeév Rudnick for stimulating discussions
and for their comments. This research was supported by the ISRAEL
SCIENCE FOUNDATION (Grant No. 1881/20).}
\begin{abstract}
We study intermediate-scale statistics for the fractional parts of
the sequence $\left(\alpha a_{n}\right)_{n=1}^{\infty}$, where $\left(a_{n}\right)_{n=1}^{\infty}$
is a positive, real-valued lacunary sequence, and $\alpha\in\mathbb{R}$.
In particular, we consider the number of elements $S_{N}\left(L,\alpha\right)$
in a random interval of length $L/N$, where $L=O\left(N^{1-\epsilon}\right)$,
and show that its variance (the number variance) is asymptotic
to $L$ with high probability w.r.t. $\alpha$, which is in agreement
with the statistics of uniform i.i.d. random points in the unit interval.
In addition, we show that the same asymptotics holds almost surely
in $\alpha\in\mathbb{R}$ when $L=O\left(N^{1/2-\epsilon}\right)$.
For slowly growing $L$, we further prove a central limit theorem
for $S_{N}\left(L,\alpha\right)$ which holds for almost all $\alpha\in\mathbb{R}$.
\end{abstract}

\maketitle

\section{Introduction}

A real-valued sequence $\left(x_{n}\right)_{n=1}^{\infty}$ is said
to be uniformly distributed (or equidistributed) modulo one, if for
every interval $I\subseteq[0,1)$, we have
\[
\lim_{N\to\infty}\frac{1}{N}\#\left\{ 1\le n\le N:\,\left\{ x_{n}\right\} \in I\right\} =\left|I\right|,
\]
where $\left\{ x\right\} $ denotes the fractional part of $x$, and
$\left|I\right|$ denotes the length of the interval $I$. There are
many examples of sequences which satisfy this property, e.g., the
Kronecker sequence $x_{n}=\alpha n$ where $\alpha$ is irrational,
and more generally (as was shown by Weyl in his pioneering 1916 paper
\cite{Weyl}) the sequence $x_{n}=\alpha_{d}n^{d}+\dots+\alpha_{1}n+\alpha_{0}$
($\alpha_{i}\in\mathbb{R}$), where at least one of the coefficients
$\alpha_{1},\dots,\alpha_{d}$ is irrational. In the metric sense,
more can be said: Weyl proved \cite{Weyl} that for \emph{any} sequence
$\left(a_{n}\right)_{n=1}^{\infty}$ of distinct integers, the sequence
$x_{n}=\alpha a_{n}$ is uniformly distributed modulo one for (Lebesgue)
almost all $\alpha\in\mathbb{R}$. This is also true for real-valued
sequences whose elements are sufficiently separated from each other
(see, e.g., \cite[Chapter 1, Corollary 4.1]{Kuipers-Niederreiter}):
if $\left(a_{n}\right)_{n=1}^{\infty}$ is a real-valued sequence,
and there exists a positive constant $\delta>0$ such that $\left|a_{n}-a_{m}\right|\ge\delta$
for each $n\ne m$, then the sequence $x_{n}=\alpha a_{n}$ is uniformly
distributed modulo one for almost all $\alpha\in\mathbb{R}$. This
condition clearly holds for real-valued, positive, \emph{lacunary}
sequences, i.e., sequences such that $a_{1}>0$, and there exists
a constant $C>1$ such that for all $n\ge1$ we have
\[
a_{n+1}\ge Ca_{n}.
\]

While the classical theory deals with the distribution of sequences
modulo one at ``large'' scales, there has been a growing interest
in recent years in the fluctuations of sequences at smaller scales.
For many sequences, it is conjectured (backed up by numerical evidence)
that the small-scale statistics (at the scale $1/N$ -- the mean
gap of the first $N$ elements of the sequence modulo one) are in
agreement with the statistics of uniform i.i.d. random points in the
unit interval (Poissonian statistics), thus demonstrating pseudo-random
behaviour for such sequences. A popular small-scale statistic is the
(normalized) gap distribution of the re-ordered first $N$ elements
of the sequence modulo one, which for many sequences is expected to
converge to the exponential distribution (``Poissonian gap statistics'')
-- the almost sure limiting distribution of the gaps in the random
model. Lacunary sequences are among the rare examples where such behaviour
has been rigorously proved to hold (in the metric sense): Rudnick
and Zaharescu proved \cite{Rudnick-Zaharescu2} Poissonian gap statistics
for almost all $\alpha\in\mathbb{R}$ for the sequence $x_{n}=\alpha a_{n}$
where $\left(a_{n}\right)_{n=1}^{\infty}$ is an integer-valued lacunary
sequence; this was recently extended to \emph{real-valued} lacunary
sequences by Chaubey and the author \cite{ChaubeyYesha}. 

Statistics in the ``mesoscopic'' regime, i.e., at the scale $L/N$,
where $L=L\left(N\right)\to\infty$ and $L=o\left(N\right)$ as $N\to\infty$,
provide further information which may capture some interesting features
of sequences. An example of such a statistic is the\emph{ number variance}
(the variance of the number of elements in random intervals, see the
definition below in our setting), famously studied for the zeros of
the Riemann zeta function, for which at small scales the number variance
is consistent with that of the eigenvalues of random matrices drawn
from the Gaussian unitary ensemble (GUE), whereas ``saturation''
occurs at larger scales (see \cite{Berry}). In the context of sequences
modulo one, only a few results have been established so far in the
mesoscopic regime, mainly concerning the leading order asymptotics
of the long-range correlations of the sequence $x_{n}=\alpha n^{2}$
(see \cite{Technau-Walker,Lutsko,Hille}); nevertheless, important
intermediate-scale statistics such as the number variance have largely
remained unexplored. The goal of this paper is to study such statistics
for real-valued lacunary sequences.

Let $\left(a_{n}\right)_{n=1}^{\infty}$ be a positive, real-valued
lacunary sequence; we are interested in the distribution of the number
of elements modulo one of the sequence $\left(x_{n}\right)_{n=1}^{\infty}=\left(\alpha a_{n}\right)_{n=1}^{\infty}$
in intervals of length $L/N$ around points $x\in[0,1)$, which we
denote by
\[
S_{N}\left(L,\alpha\right)=S_{N}\left(L,\alpha\right)\left(x\right):=\sum_{j=1}^{N}\sum_{n\in\mathbb{Z}}\chi\left(\frac{\alpha a_{j}-x+n}{L/N}\right),
\]
 where $\chi=\chi_{\left[-1/2,1/2\right]}$ is the characteristic
function of the interval $\left[-1/2,1/2\right]$.

The first statistic that we will study is the number variance
\[
\Sigma_{N}^{2}\left(L,\alpha\right):=\int_{0}^{1}\left(S_{N}\left(L,\alpha\right)\left(x\right)-L\right)^{2}\,dx,
\]
i.e., the variance of $S_{N}\left(L,\alpha\right)$, where we randomize
w.r.t. the centre of the interval $x$. We would like to show that
for generic values of $\alpha\in\mathbb{R}$, we have
\begin{equation}
\Sigma_{N}^{2}\left(L,\alpha\right)=L+o\left(L\right),\label{eq:NumberVarianceAsymp}
\end{equation}
which is in agreement with the random model. In our first main result
we show that (\ref{eq:NumberVarianceAsymp}) holds with high probability
in (essentially) the full mesoscopic regime (namely, all the way up
to $L=O\left(N^{1-\epsilon}\right)$ where $\epsilon$ is arbitrarily
small). 
\begin{thm}
\label{thm:HighProbThm}Let $\epsilon>0$, and let $I$ be a bounded
interval. Assume that $L=L\left(N\right)=O\left(N^{1-\epsilon}\right)$
as $N\to\infty$. Then (\ref{eq:NumberVarianceAsymp}) holds with
high probability w.r.t. $\alpha$: for any $\delta>0$, we have
\[
\mathrm{meas}\left\{ \text{\ensuremath{\alpha\in I:}}\left|\Sigma_{N}^{2}\left(L,\alpha\right)-L\right|>\delta L\right\} =O_{\delta,\epsilon,I}\left(N^{-\epsilon/2}\right)
\]
as $N\to\infty$. 
\end{thm}

It is desirable to extend this to an almost sure statement, which
we are able to establish in a narrower regime $L=O\left(N^{1/2-\epsilon}\right)$
(along with a technical condition on the oscillations of $L$, which
clearly holds for natural choices of $L$, e.g., when $L=N^{s}$ with
$s\le1/2-\epsilon$).
\begin{thm}
\label{thm:AlmostSureTheorem}Let $\epsilon>0$, and assume that $L=L\left(N\right)=O\left(N^{1/2-\epsilon}\right)$
and that $L\left(N+1\right)-L\left(N\right)=o\left(N^{-1/2}\right)$
as $N\to\infty$. Then for almost all $\alpha\in\mathbb{R}$, we have
\[
\Sigma_{N}^{2}\left(L,\alpha\right)=L+o\left(L\right)
\]
as $N\to\infty$.
\end{thm}

For slowly growing $L$ (and under an even milder condition on its
oscillations), we will be able to establish a central limit theorem
for $S_{N}\left(L,\alpha\right)$. This would hold for example when
$L=\left(\log N\right)^{t}$ with $t>0$.
\begin{thm}
\label{thm:CLT}Let $L=L\left(N\right)\to\infty$ as $N\to\infty$
such that for all $\eta>0$ we have $L=O\left(N^{\eta}\right)$, and
assume that there exists $\epsilon>0$ such that $L\left(N+1\right)-L\left(N\right)=O\left(N^{-\epsilon}\right)$.
Then for almost all $\alpha\in\mathbb{R}$, for any $\alpha<\beta$,
we have
\[
\mathrm{meas}\left\{ x\in[0,1):\,\alpha\le\frac{S_{N}\left(L,\alpha\right)\left(x\right)-L}{\sqrt{L}}\le\beta\right\} \longrightarrow\frac{1}{\sqrt{2\pi}}\int_{\alpha}^{\beta}e^{-\frac{t^{2}}{2}}\,dt
\]
as $N\to\infty$.
\end{thm}

We would like to stress the difference between (\ref{eq:NumberVarianceAsymp})
and some weaker notions of long-range Poissonian correlations, as
studied, e.g., in \cite{Technau-Walker,Lutsko,Hille}. Note that the
number variance $\Sigma_{N}^{2}\left(L,\alpha\right)$ can be expressed
in terms of the pair correlation function. Indeed, a direct calculation
shows (see, e.g., \cite{Marklof}) that 
\begin{equation}
\Sigma_{N}^{2}\left(L,\alpha\right)=L-L^{2}+LR_{N}^{2}\left(L,\alpha,\Delta\right)\label{eq:SigNtoRN}
\end{equation}
where
\begin{equation}
R_{N}^{2}\left(L,\alpha,\Delta\right)=\frac{1}{N}\sum_{i\ne j=1}^{N}\sum_{n\in\mathbb{Z}}\Delta\left(\frac{\alpha a_{i}-\alpha a_{j}+n}{L/N}\right)\label{eq:pair_corr_fun}
\end{equation}
is the (scaled) pair correlation function of $\left(\alpha a_{n}\right)_{n=1}^{\infty}$
with respect to the test function 
\[
\Delta=\max\left\{ 1-\left|x\right|,0\right\} .
\]
Hence, (\ref{eq:NumberVarianceAsymp}) is equivalent to 
\begin{equation}
R_{N}^{2}\left(L,\alpha,\Delta\right)=L+o\left(1\right).\label{eq:PairCorrelationAsymp}
\end{equation}

We thus see that (\ref{eq:PairCorrelationAsymp}), and therefore (\ref{eq:NumberVarianceAsymp}),
is a significantly stronger statement then long-range Poissonian pair
correlation in the sense of $R_{N}^{2}\left(L,\alpha,\Delta\right)=L+o\left(L\right)$,
where the error term is insufficient for determining the asymptotics
of the number variance. Similarly, for $k\ge2$, consider the $k$-level
correlation function
\begin{equation}
R_{N}^{k}\left(L,\alpha,\Delta\right)=\frac{1}{N}\sum_{\substack{j_{1},\dots,j_{k}=1\\
\mathrm{distinct}
}
}^{N}\sum_{n_{1},\dots,n_{k-1}\in\mathbb{Z}}\Delta\left(\frac{\alpha a_{j_{1}}-\alpha a_{j_{k}}+n_{1}}{L/N},\dots,\frac{\alpha a_{j_{k-1}}-\alpha a_{j_{k}}+n_{k-1}}{L/N}\right);\label{eq:k-level_corr}
\end{equation}
Proposition \ref{prop:MainCLTProp}, which is the main ingredient
in the proof of Theorem \ref{thm:CLT}, is notably stronger than long-range
Poissonian higher correlations in the sense of $R_{N}^{k}\left(L,\alpha,\Delta\right)=L^{k}+o\left(L^{k}\right)$,
which would be insufficient for concluding Theorem \ref{thm:CLT}.
We also remark that while the condition $L\left(N\right)\to\infty$
in Theorem \ref{thm:CLT} is essential, Theorems \ref{thm:HighProbThm}
and \ref{thm:AlmostSureTheorem} also hold for fixed $L$, thus extending
the results of \cite{Rudnick-Technau}.

\section{The number variance}

By the Poisson summation formula, we have the following identity for
the pair correlation function (\ref{eq:pair_corr_fun})
\begin{equation}
R_{N}^{2}\left(L,\alpha,\Delta\right)=L-\frac{L}{N}+T_{N}\left(L,\alpha\right),\label{eq:RNtoTN}
\end{equation}
where
\[
T_{N}\left(L,\alpha\right)=\frac{L}{N^{2}}\sum_{i\ne j=1}^{N}\sum_{0\ne n\in\mathbb{Z}}\widehat{\Delta}\left(\frac{nL}{N}\right)e\left(n\alpha\left(a_{i}-a_{j}\right)\right)
\]
with 
\[
\widehat{\Delta}\left(x\right)=\frac{\sin^{2}\left(\pi x\right)}{\pi^{2}x^{2}}
\]
(we have used the standard notation $e\left(z\right)=e^{2\pi iz}).$

We fix use a smooth, compactly supported, non-negative weight function
$\rho\in C_{c}^{\infty}\left(\mathbb{R}\right),\,\rho\ge0$, and denote
the weighted $L^{2}-$norm of $T_{n}\left(L,\alpha\right)$ by
\begin{align*}
V_{N}\left(L\right) & =\int\left|T_{N}\left(L,\alpha\right)\right|^{2}\rho\left(\alpha\right)\,d\alpha=\frac{L^{2}}{N^{4}}\sum_{\substack{0\ne n_{1}\in\mathbb{Z}\\
0\ne n_{2}\in\mathbb{Z}
}
}\widehat{\Delta}\left(\frac{n_{1}L}{N}\right)\widehat{\Delta}\left(\frac{n_{2}L}{N}\right)w\left(n_{1},n_{_{2}},N\right)
\end{align*}
where
\[
w\left(n_{1},n_{_{2}},N\right)=\sum_{\substack{1\le x_{1}\ne y_{1}\le N\\
1\le x_{2}\ne y_{2}\le N
}
}\widehat{\rho}\left(n_{1}\left(a_{x_{1}}-a_{y_{1}}\right)-n_{2}\left(a_{x_{2}}-a_{y_{2}}\right)\right);
\]
the goal of the rest of this section is to give an upper bound for
$V_{N}\left(L\right)$.

We first observe the following identity which will be useful for estimating
sums involving $\widehat{\Delta}$. 
\begin{lem}
\label{lem:Fourier_sum}Let $1\le L<N$. We have
\[
\sum_{n\in\mathbb{Z}}\widehat{\Delta}\left(\frac{nL}{N}\right)=\frac{N}{L}.
\]
\end{lem}

\begin{proof}
Let $\Delta_{N/L}\left(x\right)=\Delta\left(\frac{N}{L}x\right)$.
By the Poisson summation formula we have
\begin{align*}
\sum_{n\in\mathbb{Z}}\widehat{\Delta}\left(\frac{nL}{N}\right) & =\frac{N}{L}\sum_{n\in\mathbb{Z}}\widehat{\Delta_{N/L}}\left(n\right)=\frac{N}{L}\sum_{n\in\mathbb{Z}}\Delta_{N/L}\left(n\right)=\frac{N}{L}\Delta\left(0\right)=\frac{N}{L}.
\end{align*}
\end{proof}
In the next lemma, we will see that up to an error term of order $O\left(N^{-1}\right),$
the ranges of the summations defining $V_{N}\left(L\right)$ can be
significantly restricted. 
\begin{lem}
\label{lem:RestrictedSum}Let $1\le L<N$ and let $\text{\ensuremath{\epsilon>0}}$.
We have
\begin{align*}
V_{N}\left(L\right) & =\frac{L^{2}}{N^{4}}\sum_{\substack{0\ne\left|n_{1}\right|\le N^{4}\\
0\ne\left|n_{2}\right|\le N^{4}
}
}\widehat{\Delta}\left(\frac{n_{1}L}{N}\right)\widehat{\Delta}\left(\frac{n_{2}L}{N}\right)\tilde{w}\left(n_{1},n_{2},N\right)+O\left(N^{-1}\right)
\end{align*}
where
\[
\tilde{w}\left(n_{1},n_{2},N\right)=\sum_{\substack{1\le x_{1}\ne y_{1}\le N\\
1\le x_{2}\ne y_{2}\le N\\
\max\left\{ x_{1},x_{2},y_{1},y_{2}\right\} >N^{1/4}\\
\left|n_{1}\left(a_{x_{1}}-a_{y_{1}}\right)-n_{2}\left(a_{x_{2}}-a_{y_{2}}\right)\right|\le N^{\epsilon}
}
}\widehat{\rho}\left(n_{1}\left(a_{x_{1}}-a_{y_{1}}\right)-n_{2}\left(a_{x_{2}}-a_{y_{2}}\right)\right).
\]
\end{lem}

\begin{proof}
We have
\begin{gather*}
w\left(n_{1},n_{_{2}},N\right)-\tilde{w}\left(n_{1},n_{2},N\right)\ll\sum_{\substack{x_{1}\ne y_{1}\ge1\\
x_{2}\ne y_{2}\ge1\\
\max\left\{ x_{1},x_{2},y_{1},y_{2}\right\} \le N^{1/4}
}
}\left|\widehat{\rho}\left(n_{1}\left(a_{x_{1}}-a_{y_{1}}\right)-n_{2}\left(a_{x_{2}}-a_{y_{2}}\right)\right)\right|\\
+\sum_{\substack{1\le x_{1}\ne y_{1}\le N\\
1\le x_{2}\ne y_{2}\le N\\
\left|n_{1}\left(a_{x_{1}}-a_{y_{1}}\right)-n_{2}\left(a_{x_{2}}-a_{y_{2}}\right)\right|>N^{\epsilon}
}
}\left|\widehat{\rho}\left(n_{1}\left(a_{x_{1}}-a_{y_{1}}\right)-n_{2}\left(a_{x_{2}}-a_{y_{2}}\right)\right)\right|\ll N,
\end{gather*}
where we bounded the first summation using the bound $\widehat{\rho}\ll1$
and the second summation using $\widehat{\rho}\left(x\right)\ll x^{-k}$
for all $k>0$. Thus,
\begin{align*}
\frac{L^{2}}{N^{4}}\sum_{\substack{0\ne n_{1}\in\mathbb{Z}\\
0\ne n_{2}\in\mathbb{Z}
}
}\widehat{\Delta}\left(\frac{n_{1}L}{N}\right)\widehat{\Delta}\left(\frac{n_{2}L}{N}\right)\left(w\left(n_{1},n_{_{2}},N\right)-\tilde{w}\left(n_{1},n_{2},N\right)\right) & \ll\frac{L^{2}}{N^{3}}\left(\sum_{n\in\mathbb{Z}}\widehat{\Delta}\left(\frac{nL}{N}\right)\right)^{2}=\frac{1}{N}
\end{align*}
where in the last equality we used Lemma \ref{lem:Fourier_sum}. Finally,
by bounding $\tilde{w}$ trivially and applying the bound $\widehat{\Delta}\left(x\right)\ll x^{-2}$,
we have
\begin{gather*}
\frac{L^{2}}{N^{4}}\sum_{\substack{0\ne n_{1}\in\mathbb{Z}\\
0\ne n_{2}\in\mathbb{Z}\\
\max\left\{ \left|n_{1}\right|,\left|n_{2}\right|\right\} >N^{4}
}
}\widehat{\Delta}\left(\frac{n_{1}L}{N}\right)\widehat{\Delta}\left(\frac{n_{2}L}{N}\right)\tilde{w}\left(n_{1},n_{2},N\right)\ll L^{2}\sum_{m>N^{4}}\widehat{\Delta}\left(\frac{mL}{N}\right)\sum_{n\in\mathbb{Z}}\widehat{\Delta}\left(\frac{nL}{N}\right)\\
=NL\sum_{m>N^{4}}\widehat{\Delta}\left(\frac{mL}{N}\right)\ll\frac{N^{3}}{L}\sum_{m>N^{4}}m^{-2}\ll\frac{1}{LN}
\end{gather*}
which concludes the proof.
\end{proof}
We will now analyze when the summation defining $\tilde{w}$ does
not vanish.

\begin{prop}
\label{prop:MainProp}Fix $n_{1}$ such that $0<\left|n_{1}\right|\le N^{4}$,
and $x_{1},y_{1}$ such that $1\le y_{1}<x_{1}\le N$, $x_{1}>N^{1/4}$.
Then there exist at most $O\left(N^{\epsilon}\log N\right)$ values
of $n_{2},x_{2},y_{2}$ such that $0<\left|n_{2}\right|\le N^{4}$,
$x_{2}\le x_{1},$ $1\le y_{2}<x_{2}\le N$, and 
\begin{align}
\left|n_{1}\left(a_{x_{1}}-a_{y_{1}}\right)-n_{2}\left(a_{x_{2}}-a_{y_{2}}\right)\right| & \le N^{\epsilon}.\label{eq:SmallProduct}
\end{align}
\end{prop}

\begin{proof}
We follow the ideas of \cite{Rudnick-Zaharescu,Rudnick-Technau}.
We have
\[
\left|n_{1}\right|\left(a_{x_{1}}-a_{y_{1}}\right)\ge a_{x_{1}}-a_{x_{1}-1}=a_{x_{1}}\left(1-\frac{a_{x_{1}-1}}{a_{x_{1}}}\right)\ge\left(1-\frac{1}{C}\right)a_{x_{1}};
\]
on the other hand,
\[
\left|n_{2}\right|\left(a_{x_{2}}-a_{y_{2}}\right)\le N^{4}a_{x_{2}}=N^{4}a_{x_{1}}\frac{a_{x_{2}}}{a_{x_{1}}}\le a_{x_{1}}\frac{N^{4}}{C^{x_{1}-x_{2}}}.
\]
Substituting in (\ref{eq:SmallProduct}), we obtain

\[
1-\frac{1}{C}-\frac{N^{4}}{C^{x_{1}-x_{2}}}\le N^{\epsilon}a_{x_{1}}^{-1}.
\]
Since $x_{1}>N^{1/4}$, we have $a_{x_{1}}\ge a_{1}C^{x_{1}-1}>a_{1}C^{N^{1/4}-1}$,
and hence
\[
1-\frac{1}{C}-\frac{N^{4}}{C^{x_{1}-x_{2}}}\le N^{\epsilon}a_{1}^{-1}C^{-\left(N^{1/4}-1\right)},
\]
and therefore for sufficiently large $N$ we have
\[
C^{x_{1}-x_{2}}\le\frac{N^{4}}{1-\frac{1}{C}-N^{\epsilon}a_{1}^{-1}C^{-\left(N^{1/4}-1\right)}}\ll N^{4}
\]
so that $x_{1}-x_{2}\ll\log N$. Thus, there are at most $O\left(\log N\right)$
possible values for $x_{2}$, and moreover we conclude that $x_{2}\gg N^{1/4}$,
and hence $a_{x_{2}}\gg C^{N^{1/4}}.$

We now fix $x_{2}$. Since
\[
\left|n_{2}-n_{1}\frac{a_{x_{1}}-a_{y_{1}}}{a_{x_{2}}-a_{y_{2}}}\right|\le\frac{N^{\epsilon}}{a_{x_{2}}-a_{y_{2}}}\le\frac{N^{\epsilon}}{\left(1-\frac{1}{C}\right)a_{x_{2}}}\ll\frac{N^{\epsilon}}{C^{N^{1/4}}},
\]
we see that for sufficiently large $N$, the integer $n_{2}$ is uniquely
determined by the values of $x_{1},x_{2},y_{1},y_{2},n_{1}$. It is
therefore sufficient to bound the number of possible values of $y_{2}$.
There are $O\left(\log N\right)$ values of $y_{2}$ such that $x_{2}-y_{2}\le5\log_{C}N$.
We will therefore count the number of possible values of $y_{2}$
such that $x_{2}-y_{2}>5\log_{C}N$. For such $y_{2}$ we have
\[
a_{y_{2}}=a_{x_{2}}\frac{a_{y_{2}}}{a_{x_{2}}}\le\frac{a_{x_{2}}}{C^{x_{2}-y_{2}}}<\frac{a_{x_{2}}}{N^{5}}
\]
and therefore
\begin{align*}
n_{1}\left(a_{x_{1}}-a_{y_{1}}\right) & =n_{2}\left(a_{x_{2}}-a_{y_{2}}\right)+O\left(N^{\epsilon}\right)=n_{2}a_{x_{2}}\left(1-\frac{a_{y_{2}}}{a_{x_{2}}}+O\left(\frac{N^{\epsilon}}{C^{N^{1/4}}}\right)\right)\\
 & =n_{2}a_{x_{2}}\left(1+O\left(N^{-5}\right)\right).
\end{align*}
Hence, given $\left(y_{2},n_{2}\right)$ and $\left(y_{2}',n_{2}'\right)$
such that $x_{2}-y_{2}>5\log_{C}N$ and $x_{2}-y_{2}'>5\log_{C}N$,
we have
\[
n_{2}a_{x_{2}}\left(1+O\left(N^{-5}\right)\right)=n_{2}'a_{x_{2}}\left(1+O\left(N^{-5}\right)\right)
\]
and since $\left|n_{2}\right|\le N^{4}$ we conclude that
\[
n_{2}'=n_{2}+O\left(N^{-1}\right)
\]
so in fact $n_{2}'=n_{2}$. We therefore see that the value of $n_{2}$
is identical for each $y_{2}$ such that $x_{2}-y_{2}>5\log_{C}N$.
But, for such $y_{2}$, (\ref{eq:SmallProduct}) gives
\[
a_{y_{2}}\in\left[a_{x_{2}}-\frac{n_{1}\left(a_{x_{1}}-a_{y_{1}}\right)}{n_{2}}-\frac{N^{\epsilon}}{n_{2}},a_{x_{2}}-\frac{n_{1}\left(a_{x_{1}}-a_{y_{1}}\right)}{n_{2}}+\frac{N^{\epsilon}}{n_{2}}\right]
\]
so that $a_{y_{2}}$ lies in an interval of length $O\left(N^{\epsilon}\right)$,
and since
\[
a_{n+1}-a_{n}=a_{n+1}\left(1-\frac{a_{n}}{a_{n+1}}\right)\ge a_{n+1}\left(1-\frac{1}{C}\right)\gg1
\]
there could be at most $O\left(N^{\epsilon}\right)$ values of $y_{2}$
in this interval. 
\end{proof}
As an immediate corollary of Lemma \ref{lem:RestrictedSum} and Proposition
\ref{prop:MainProp}, we obtain an upper bound for $V_{N}\left(L\right)$.
\begin{cor}
Let $1\le L<N$ and let $\text{\ensuremath{\epsilon>0}}$. We have
\begin{equation}
V_{N}\left(L\right)=O\left(LN^{-1+\epsilon}\right).\label{eq:VarBound}
\end{equation}
\end{cor}

\begin{proof}
We use the bound $\widehat{\rho}\ll1$ and Lemma \ref{lem:RestrictedSum}
to conclude that
\[
V_{N}\left(L\right)\ll\frac{L^{2}}{N^{4}}\sum_{\substack{0\ne\left|n_{1}\right|\le N^{4}\\
1\le x_{1}\ne y_{1}\le N
}
}\widehat{\Delta}\left(\frac{n_{1}L}{N}\right)\sum_{\substack{0\ne\left|n_{2}\right|\le N^{4}\\
1\le x_{2}\ne y_{2}\le N\\
\max\left\{ x_{1},x_{2},y_{1},y_{2}\right\} >N^{1/4}\\
\left|n_{1}\left(a_{x_{1}}-a_{y_{1}}\right)-n_{2}\left(a_{x_{2}}-a_{y_{2}}\right)\right|\le N^{\epsilon/2}
}
}\widehat{\Delta}\left(\frac{n_{2}L}{N}\right)+N^{-1}.
\]
By symmetry we can assume that $y_{1}<x_{1},$ $x_{2}\le x_{1}$ and
$y_{2}<x_{2}$, so that by the bound $\widehat{\Delta}\ll1$ and by
Proposition \ref{prop:MainProp}, the inner summation is $O\left(N^{\epsilon/2}\log N\right)$.
Hence, Lemma \ref{lem:Fourier_sum} gives the required bound (\ref{eq:VarBound}).
\end{proof}

\section{Proofs of Theorems \ref{thm:HighProbThm} and \ref{thm:AlmostSureTheorem}}

We are now ready to prove Theorem \ref{thm:HighProbThm}.
\begin{proof}[Proof of Theorem \ref{thm:HighProbThm} ]
By (\ref{eq:SigNtoRN}) and (\ref{eq:RNtoTN}) we have
\[
\frac{\Sigma_{N}^{2}\left(L,\alpha\right)-L}{L}=T_{N}\left(L,\alpha\right)-\frac{L}{N}.
\]
Hence, for sufficiently large $N$ we have
\begin{align*}
\text{meas}\left\{ \text{\ensuremath{\alpha\in I:}}\left|\Sigma_{N}^{2}\left(L,\alpha\right)-L\right|>\delta L\right\}  & =\text{meas}\left\{ \text{\ensuremath{\alpha\in I:}}\left|T_{N}\left(L,\alpha\right)-\frac{L}{N}\right|>\delta\right\} \\
 & \le\text{meas}\left\{ \text{\ensuremath{\alpha\in I:}}\left|T_{N}\left(L,\alpha\right)\right|>\delta/2\right\} .
\end{align*}
Fix a smooth, compactly supported, weight function $\rho\in C_{c}^{\infty}\left(\mathbb{R}\right)$
such that $1_{I}\left(x\right)\le\rho\left(x\right)$. By Chebyshev's
inequality we conclude that for sufficiently large $N$ we have
\begin{align}
\text{meas}\left\{ \text{\ensuremath{\alpha\in I:}}\left|\Sigma_{N}^{2}\left(L,\alpha\right)-L\right|>\delta L\right\}  & \le\frac{4\int_{I}\left|T_{N}\left(L,\alpha\right)\right|^{2}\,d\alpha}{\delta^{2}}\le\frac{4V_{N}\left(L\right)}{\delta^{2}}\label{eq:HighProbEst}\\
 & \ll_{\delta,\epsilon,I}LN^{-1+\epsilon/2}\ll N^{-\epsilon/2}\nonumber 
\end{align}
where we used (\ref{eq:VarBound}) with $\epsilon/2$.
\end{proof}
We will now turn to prove Theorem \ref{thm:AlmostSureTheorem}, that
is, we will show that (\ref{eq:NumberVarianceAsymp}) (or equivalently
(\ref{eq:PairCorrelationAsymp})) holds for almost all $\alpha\in\mathbb{R}$.
It sufficient to prove this for $\alpha\in I$ where $I$ is a bounded
interval. We first show that almost sure convergence of the pair correlation
holds along a subsequence.
\begin{lem}
\label{lem:SubsequenceLemma}Let I be a bounded interval, and let
$\epsilon>0$. Assume that $L=L\left(N\right)=O\left(N^{1/2-\epsilon}\right)$
as $N\to\infty$. Let $N_{m}=m^{2}$, and denote $L_{m}=L\left(N_{m}\right)$.
Then for almost all $\alpha\in I$, we have
\begin{equation}
R_{N_{m}}^{2}\left(L_{m},\alpha,\Delta\right)=L_{m}+o\left(1\right)\label{eq:SubSequenceAsymp}
\end{equation}
as $m\to\infty$.
\end{lem}

\begin{proof}
Applying (\ref{eq:VarBound}) as in (\ref{eq:HighProbEst}), for every
$\delta>0$ and $N$ sufficiently large we have 
\begin{align*}
\text{meas}\left\{ \text{\ensuremath{\alpha\in I:}}\left|R_{N}^{2}\left(L,\alpha,\Delta\right)-L\right|>\delta\right\}  & \le\text{meas}\left\{ \text{\ensuremath{\alpha\in I:}}\left|T_{N}\left(L,\alpha\right)\right|>\delta/2\right\} \\
 & \ll_{\delta,\epsilon,I}LN^{-1+\epsilon/2}\ll N^{-1/2-\epsilon/2}.
\end{align*}
so that
\begin{equation}
\text{meas}\left\{ \text{\ensuremath{\alpha\in I:}}\left|R_{N_{m}}^{2}\left(L_{m},\alpha,\Delta\right)-L_{m}\right|>\delta\right\} \ll_{\delta,\epsilon,I}m^{-1-\epsilon};\label{eq:BorelCantelliBound}
\end{equation}
the asymptotics (\ref{eq:SubSequenceAsymp}) thus holds for almost
all $\alpha\in I$ by the Borel-Cantelli lemma.
\end{proof}
We now have all that is needed to prove Theorem \ref{thm:AlmostSureTheorem}.
\begin{proof}[Proof of Theorem \ref{thm:AlmostSureTheorem}]
It is sufficient to show that 
\begin{equation}
R_{N}^{2}\left(L,\alpha,\Delta\right)=L+o\left(1\right)\label{eq:PairCorrelationAsympProof}
\end{equation}
as $N\to\infty$ for almost all $\alpha\in I$. Let $N_{m}=m^{2}$.
For any $N$ there exists $m$ such that $N_{m-1}\le N<N_{m}$. Moreover,
$\frac{N_{m}}{N}=1+O\left(m^{-1}\right)$, and by the assumption 
\[
L\left(N+1\right)-L\left(N\right)=o\left(N^{-1/2}\right)
\]
we have
\[
L\left(N\right)=L_{m}+o\left(\frac{N-N_{m}}{\sqrt{N}}\right)=L_{m}+o\left(1\right);
\]
thus,  there exists a constant $C>0$ such that for any $\delta>0$,
for sufficiently large $N$ we have 
\begin{align*}
R_{N}^{2}\left(L,\alpha,\Delta\right) & \le\frac{N_{m}}{N}R_{N_{m}}^{2}\left(L\cdot\frac{N_{m}}{N},\alpha,\Delta\right)\\
 & \le\left(1+Cm^{-1}\right)R_{N_{m}}^{2}\left(\left(L_{m}+\delta\right)\cdot\left(1+Cm^{-1}\right),\alpha,\Delta\right);
\end{align*}
by applying Lemma \ref{lem:SubsequenceLemma} with $\left(L+\delta\right)\cdot\left(1+CN^{-1/2}\right)$
instead of $L$, as $m\to\infty$ we have 
\begin{align*}
R_{N_{m}}^{2}\left(\left(L_{m}+\delta\right)\cdot\left(1+Cm^{-1}\right),\alpha,\Delta\right) & =\left(L_{m}+\delta\right)\cdot\left(1+Cm^{-1}\right)+o\left(1\right)\\
 & =L_{m}+\delta+o\left(1\right)
\end{align*}
for all $\alpha\in I_{\delta}$, where $I_{\delta}$ is a full measure
set in $I$ (note that $L_{m}=o\left(m\right)$ by the assumption
$L=O\left(N^{1/2-\epsilon}\right)$). Hence, for sufficiently large
$N$ we have
\begin{equation}
R_{N}^{2}\left(L,\alpha,\Delta\right)\le\left(1+Cm^{-1}\right)\left(L_{m}+\delta+o\left(1\right)\right)=L_{m}+\delta+o\left(1\right)=L+\delta+o\left(1\right)\label{eq:PairCorrelationUpperBound}
\end{equation}
for all $\alpha\in I_{\delta}$. Symmetrically, 
\begin{equation}
R_{N}^{2}\left(L,\alpha,\Delta\right)\ge L-\delta-o\left(1\right)\label{eq:PairCorrelationLowerBound}
\end{equation}
for $\alpha$ in a full measure set in $I$ (depending on $\delta$);
since $\delta>0$ can be taken arbitrarily small along a countable
sequence of values, and a countable intersection of full measure sets
is still of full measure, the bounds (\ref{eq:PairCorrelationUpperBound})
and (\ref{eq:PairCorrelationLowerBound}) imply (\ref{eq:PairCorrelationAsympProof})
for almost all $\alpha\in I$.
\end{proof}
\begin{rem*}
The faster $L$ grows, the sparser the subsequence $N_{m}$ one has
to take in order to apply the Borel-Cantelli lemma in the proof of
Lemma \ref{lem:SubsequenceLemma}. On the other hand, since we require
the condition $L_{m}=o\left(m\right),$ the subsequence $N_{m}$ cannot
be too sparse. For example, if $L=N^{s},$ and $N_{m}=\lfloor m^{t}\rfloor$,
one needs $t>\frac{1}{1-s}$ for (\ref{eq:BorelCantelliBound}) to
hold, but also $t<\frac{1}{s}$, so that $s<1/2$. This explains why
the above argument only works for $L$ growing slower than $N^{1/2}.$ 
\end{rem*}

\section{Higher order correlations -- proof of Theorem \ref{thm:CLT}\label{sec:CLT}}

Taking expectations w.r.t $x$, for all $k\ge2$, we have
\begin{align}
\mathbb{E}\left[\left(S_{N}\left(L,\alpha\right)\right)^{k}\right] & =\sum_{j_{1},\dots,j_{k}=1}^{N}\sum_{n_{1},\dots,n_{k}\in\mathbb{Z}}\int_{0}^{1}\prod_{i=1}^{k}\chi\left(\frac{\alpha a_{j_{i}}-x+n_{i}}{L/N}\right)\,dx\label{eq:kth_moment}\\
 & =\sum_{j_{1},\dots,j_{k}=1}^{N}\sum_{n_{1},\dots,n_{k-1}\in\mathbb{Z}}\int_{\mathbb{R}}\prod_{i=1}^{k-1}\chi\left(\frac{\alpha a_{j_{i}}-x+n_{i}}{L/N}\right)\chi\left(\frac{\alpha a_{j_{k}}-x}{L/N}\right)\,dx\nonumber \\
 & =\frac{L}{N}\sum_{j_{1},\dots,j_{k}=1}^{N}\sum_{n_{1},\dots,n_{k-1}\in\mathbb{Z}}\Delta\left(\frac{\alpha a_{j_{1}}-\alpha a_{j_{k}}+n_{1}}{L/N},\dots,\frac{\alpha a_{j_{k-1}}-\alpha a_{j_{k}}+n_{k-1}}{L/N}\right),\nonumber 
\end{align}
where
\begin{align*}
\Delta\left(t_{1},\dots,t_{k-1}\right) & =\int_{\mathbb{R}}\prod_{i=1}^{k-1}\chi\left(t_{i}-x\right)\chi\left(x\right)\,dx\\
 & =\max\left\{ 1-\text{\ensuremath{\max\{0,t_{1},\dots,t_{k-1}\}+\min\left\{ 0,t_{1},\dots,t_{k-1}\right\} }},0\right\} ;
\end{align*}
we have (see \cite[Lemma 13]{Hauke-Zafeiropoulos}) 
\begin{equation}
\int_{\mathbb{R}^{k-1}}\Delta\left(t_{1},\dots,t_{k-1}\right)\,dt_{1}\dots dt_{k-1}=1.\label{eq:UnitIntegral}
\end{equation}

For $0\le j\le k,$ denote by $\begin{Bmatrix}k\\
j
\end{Bmatrix}$ the Stirling number of the second kind, i.e., the number of ways
to partition a set of $k$ elements into $j$ non-empty subsets. We
partition the sum over $j_{1},\dots,j_{k}$ on the right-hand side
of (\ref{eq:kth_moment}) into sums with $j$ distinct indices. The
term corresponding to $j=1$ is clearly equal to $L$. Recalling the
definition (\ref{eq:k-level_corr}) of the $j$-level correlation
functions $R_{N}^{j}\left(L,\alpha,\Delta\right)$, we then have
\[
\mathbb{E}\left[\left(S_{N}\left(L,\alpha\right)\right)^{k}\right]=L+L\sum_{j=2}^{k}\begin{Bmatrix}k\\
j
\end{Bmatrix}R_{N}^{j}\left(L,\alpha,\Delta\right).
\]
In view of Lemma \ref{lem:PoissonToNormal}, Theorem \ref{thm:CLT}
will be a direct consequence of the following proposition.
\begin{prop}
\label{prop:MainCLTProp}Let $L=L\left(N\right)$ such that for all
$\eta>0$ we have $L=O\left(N^{\eta}\right)$, and assume that there
exists $\epsilon>0$ such that $L\left(N+1\right)-L\left(N\right)=O\left(N^{-\epsilon}\right)$.
Then for almost all $\alpha\in\mathbb{R}$, we have
\begin{equation}
R_{N}^{j}\left(L,\alpha,\Delta\right)=L^{j-1}+O\left(L^{-s}\right)\label{eq:corr_asymp}
\end{equation}
for all $j\ge2$ and all $s>0$.
\end{prop}

We apply the following strategy for proving Proposition \ref{prop:MainCLTProp}.
We first prove an analogous result with a smooth test function along
a subsequence. We then unsmooth along the subsequence, and finally
deduce the result along the full sequence. We would like to use the
results of \cite{ChaubeyYesha}, and for that it would be more convenient
to work with a ``transformed'' correlation function: for a smooth,
compactly supported function $\psi:\mathbb{R}^{k-1}\to\mathbb{R}$
and for $k\ge2$ we denote the smoothed $k$-level correlation function
\[
R_{N}^{k}\left(L,\alpha,\psi\right)=\frac{1}{N}\sum_{\substack{j_{1},\dots,j_{k}=1\\
\mathrm{distinct}
}
}^{N}\sum_{n_{1},\dots,n_{k-1}\in\mathbb{Z}}\psi\left(\frac{\alpha a_{j_{1}}-\alpha a_{j_{k}}+n_{1}}{L/N},\dots,\frac{\alpha a_{j_{k-1}}-\alpha a_{j_{k}}+n_{k-1}}{L/N}\right)
\]
and the transformed smoothed $k$-level correlation function
\[
\tilde{R}_{N}^{k}\left(\alpha,\psi\right)=\frac{1}{N}\sum_{\substack{j_{1},\dots,j_{k}=1\\
\mathrm{distinct}
}
}^{N}\sum_{n_{1},\dots,n_{k-1}\in\mathbb{Z}}\psi\left(N\left(\alpha a_{j_{1}}-\alpha a_{j_{2}}+n_{1}\right),\dots,N\left(\alpha a_{j_{k-1}}-\alpha a_{j_{k}}+n_{k-1}\right)\right).
\]
Then, for sufficiently large $N$ (recall that $L=o\left(N\right))$
we have
\begin{equation}
R_{N}^{k}\left(L,\alpha,\psi\right)=\tilde{R}_{N}^{k}\left(\alpha,\tilde{\psi}_{L}\right)\label{eq:correlation_transformation}
\end{equation}
where
\[
\tilde{\psi}_{L}\left(t_{1},\dots,t_{k-1}\right)=\psi\left(\frac{t_{1}+\dots+t_{k-1}}{L},\frac{t_{2}+\dots+t_{k-1}}{L},\dots,\frac{t_{k-1}}{L}\right).
\]
For the transformed correlation function we have the following $L^{2}-$norm
estimate: let $I$ be a bounded interval, and let
\[
V\left(\tilde{R}_{N}^{k}\left(\psi\right)\right)=\int_{I}\left(\tilde{R}_{N}^{k}\left(\alpha,\psi\right)-C_{k}\left(N\right)\widehat{\psi}\left(0\right)\right)^{2}\,d\alpha,
\]
where
\[
C_{k}\left(N\right)=\left(1-\frac{1}{N}\right)\cdots\left(1-\frac{k-1}{N}\right).
\]

\begin{lem}
\label{lem:EffectiveVar}Let $k\ge2$. For each $\eta>0$ there exists
$r=r\left(\eta\right)$ such that
\begin{equation}
V\left(\tilde{R}_{N}^{k}\left(\psi\right)\right)=O\left(\left\Vert \psi\right\Vert _{r,1}^{2}N^{-1+\eta}\right)\label{eq:VarianceBoundRk}
\end{equation}
where $\left\Vert \psi\right\Vert _{r,1}=\sum\limits _{\left|\alpha\right|\le r}\left\Vert \partial^{\alpha}\psi\right\Vert _{1}$.
\end{lem}

\begin{proof}
For a smooth, compactly supported, non-negative weight function $\rho\in C_{c}^{\infty}\left(\mathbb{R}\right),\,\rho\ge0$,
denote
\[
V\left(\tilde{R}_{N}^{k}\left(\psi\right),\rho\right)=\int_{\mathbb{R}}\left(\tilde{R}_{N}^{k}\left(\alpha,\psi\right)-C_{k}\left(N\right)\widehat{\psi}\left(0\right)\right)^{2}\rho\left(\alpha\right)\,d\alpha.
\]
Then Proposition 7 in \cite{ChaubeyYesha} implies that for each $\eta>0$
there exists $r=r\left(\eta\right)$ such that
\begin{equation}
V\left(\tilde{R}_{N}^{k}\left(\psi\right),\rho\right)=O\left(\left\Vert \psi\right\Vert _{r,1}^{2}N^{-1+\eta}\right);\label{eq:VarianceEffectiveBound}
\end{equation}
while the constant $\left\Vert \psi\right\Vert _{r,1}^{2}$ is not
explicitly stated there, it follows from the proof, which we now sketch
(for the full details we refer the reader to \cite{ChaubeyYesha}):
for $x=\left(x_{1},\dots,x_{k}\right),$ denote
\[
\Delta_{\left(a_{n}\right)}\left(x\right)=\left(a_{x_{1}}-a_{x_{2}},\dots,a_{x_{k-1}}-a_{x_{k}}\right);
\]
by the Poisson summation formula, we have
\[
\tilde{R}_{N}^{k}\left(\alpha,\psi\right)=C_{k}\left(N\right)\widehat{\psi}\left(0\right)+\frac{1}{N^{k}}\sum_{0\ne n\in\mathbb{Z}^{k-1}}\widehat{\psi}\left(\frac{n}{N}\right)\sum_{\substack{x=\left(x_{1},\dots,x_{k}\right)\\
1\le x_{1},\dots,x_{k}\le N\,\mathrm{distinct}
}
}e\left(\alpha n\cdot\Delta_{\left(a_{n}\right)}\left(x\right)\right),
\]
and hence
\[
V\left(\tilde{R}_{N}^{k}\left(\psi\right),\rho\right)=\frac{1}{N^{2k}}\sum_{0\ne n,m\in\mathbb{Z}^{k-1}}\widehat{\psi}\left(\frac{n}{N}\right)\overline{\widehat{\psi}\left(\frac{m}{N}\right)}\sum^{*}\widehat{\rho}\left(n\cdot\Delta_{\left(a_{n}\right)}\left(x\right)-m\cdot\Delta_{\left(a_{n}\right)}\left(y\right)\right).
\]
where the range of the summation $\sum\limits ^{*}$ is over $x=\left(x_{1},\dots,x_{k}\right)$
where $1\le x_{1},\dots,x_{k}\le N$ are distinct, and $y=\left(y_{1},\dots,y_{k}\right)$
where $1\le y_{1},\dots,y_{k}\le N$ are distinct. Fix $\epsilon>0$;
by splitting the summation over $n,m$ into different ranges and using
the bounds $\widehat{\rho}\ll1$, $\left|\widehat{\psi}\right|\le\left\Vert \psi\right\Vert _{1}\le\left\Vert \psi\right\Vert _{r,1}$
and $\left|\widehat{\psi}\right|\ll\left\Vert \psi\right\Vert _{r,1}\left\Vert x\right\Vert _{\infty}^{-r}$
(for arbitrarily large $r$), we obtain
\begin{align*}
V\left(\tilde{R}_{N}^{k}\left(\psi\right),\rho\right) & \ll\left\Vert \psi\right\Vert _{r,1}^{2}\Bigl(\frac{1}{N^{2k}}\sum_{\left\Vert n\right\Vert _{\infty},\left\Vert m\right\Vert _{\infty}>N^{1+\epsilon}}\left\Vert \frac{n}{N}\right\Vert _{\infty}^{-r}\left\Vert \frac{m}{N}\right\Vert _{\infty}^{-r}\sum^{*}1\\
 & +\frac{1}{N^{2k}}\sum_{\left\Vert n\right\Vert _{\infty}>N^{1+\epsilon},0<\left\Vert m\right\Vert _{\infty}\le N^{1+\epsilon}}\left\Vert \frac{n}{N}\right\Vert _{\infty}^{-r}\sum^{*}1\\
 & +\frac{1}{N^{2k}}\sum_{0<\left\Vert n\right\Vert _{\infty},\left\Vert m\right\Vert _{\infty}\le N^{1+\epsilon}}\sum^{*}\left|\widehat{\rho}\left(n\cdot\Delta_{\left(a_{n}\right)}\left(x\right)-m\cdot\Delta_{\left(a_{n}\right)}\left(y\right)\right)\right|\Bigr).
\end{align*}
The contribution of the first two terms is negligible by a trivial
estimate, and so is the contribution of the third term restricted
to the range $\left|n\cdot\Delta_{\left(a_{n}\right)}\left(x\right)-m\cdot\Delta_{\left(a_{n}\right)}\left(y\right)\right|>N^{\epsilon}$
(choosing $r$ sufficiently large depending on $\epsilon$). The rest
of the contribution from the third term is then bounded by \cite[Proposition 2]{ChaubeyYesha}
which states that there are at most $O\left(N^{2k-1+4k\epsilon}\right)$
values of $n,m,x,y$ in the above ranges such that $\left|n\cdot\Delta_{\left(a_{n}\right)}\left(x\right)-m\cdot\Delta_{\left(a_{n}\right)}\left(y\right)\right|\le N^{\epsilon}$,
which gives (\ref{eq:VarianceEffectiveBound}). \\
Finally, if we choose $\rho$ such that $\rho\ge1_{I}$, then
\[
V\left(\tilde{R}_{N}^{k}\left(\psi\right)\right)\le V\left(\tilde{R}_{N}^{k}\left(\psi\right),\rho\right),
\]
and (\ref{eq:VarianceBoundRk}) follows.
\end{proof}
Fix $\eta>0$, and let $r=r\left(\eta\right)>1$ be as in Lemma \ref{lem:EffectiveVar};
let $0<\epsilon<1$, $\delta=N^{-\frac{\epsilon}{8r}}$, and assume
that $\psi\in C_{c}^{\infty}\left(\mathbb{R}^{k-1}\right)$ is a smooth
approximation to $\Delta$, such that $\left\Vert \Delta-\psi\right\Vert _{\infty}\ll\delta$
and such that $\left\Vert \psi\right\Vert _{r,1}\ll\delta^{-r}$.
By (\ref{eq:UnitIntegral}), if $L$ grows slower than any power of
$N$, then
\[
\widehat{\tilde{\psi}_{L}}\left(0\right)=L^{k-1}\widehat{\psi}\left(0\right)=L^{k-1}+O\left(\delta L^{k-1}\right)=L^{k-1}+O\left(N^{-\frac{\epsilon}{8r}+\eta}\right).
\]
 Moreover, we have 
\[
\left\Vert \tilde{\psi}_{L}\right\Vert _{r,1}^{2}\ll L^{2\left(k-1\right)}\left\Vert \psi\right\Vert _{r,1}^{2}\ll L^{2\left(k-1\right)}\delta^{-2r}\ll N^{\frac{\epsilon}{4}+\eta}.
\]

We deduce almost sure convergence along a subsequence.
\begin{lem}
\label{lem:k-corr_subsequence}Let $L=L\left(N\right)$ be such that
for all $\eta>0$ we have $L=O\left(N^{\eta}\right)$ and let $k\ge2$.
Let $0<\epsilon<1$, $N_{m}=\lfloor m^{1+\epsilon}\rfloor$, and denote
$L_{m}=L\left(N_{m}\right)$. Then for almost all $\alpha\in I$,
we have
\begin{equation}
R_{N_{m}}^{k}\left(L_{m},\alpha,\Delta\right)=L_{m}^{k-1}+O\left(L_{m}^{-s}\right)\label{eq:RkAlongSubseq}
\end{equation}
for all $s>0,$ as $m\to\infty$, 
\end{lem}

\begin{proof}
It is sufficient to show that for any fixed $s>0$, (\ref{eq:RkAlongSubseq})
holds for almost all $\alpha\in I$. By identity (\ref{eq:correlation_transformation}),
Lemma \ref{lem:EffectiveVar}, and the upper bound on $L$, for each
$\eta>0$ there exists $r=r\left(\eta\right)$ such that
\begin{gather*}
\int_{I}\left(L_{m}^{s}\left(R_{N_{m}}^{k}\left(L_{m},\alpha,\psi\right)-C_{k}\left(N_{m}\right)\widehat{\tilde{\psi}_{L_{m}}}\left(0\right)\right)\right)^{2}\,d\alpha=L_{m}^{2s}V\left(\tilde{R}_{N_{m}}^{k}\left(\tilde{\psi}_{L_{m}}\right)\right)\\
=O\left(\left\Vert \tilde{\psi}_{L_{m}}\right\Vert _{r,1}^{2}L_{m}^{2s}N_{m}^{-1+\eta}\right)=O\left(N_{m}^{-1+\frac{\epsilon}{4}+3\eta}\right)=O\left(m^{\left(1+\epsilon\right)\left(-1+\frac{\epsilon}{4}+3\eta\right)}\right).
\end{gather*}
Hence, by the Borel-Cantelli lemma, for $\eta$ sufficiently small
we have
\[
L_{m}^{s}\left(R_{N_{m}}^{k}\left(L_{m},\alpha,\psi\right)-C_{k}\left(N_{m}\right)\widehat{\tilde{\psi}_{L_{m}}}\left(0\right)\right)=o\left(1\right)
\]
for almost all $\alpha\in I$, and in particular
\begin{equation}
R_{N_{m}}^{k}\left(L_{m},\alpha,\psi\right)=C_{k}\left(N_{m}\right)\widehat{\tilde{\psi}_{L_{m}}}\left(0\right)+O\left(L_{m}^{-d}\right)=L_{m}^{k-1}+O\left(L_{m}^{-s}\right),\label{eq:Smooth_Subseq_R_k_asymp}
\end{equation}
where we used again the upper bound on $L$.

Let $\psi=\psi_{\pm}$ be approximations to $\Delta$ satisfying the
above assumptions such that $\psi_{-}\le\Delta\le\psi_{+}$; a simple
way to construct such approximations is to convolve the functions
\[
\Delta_{\delta}^{\pm}\left(t_{1},\dots,t_{k-1}\right):=\max\left\{ 1\pm\delta-\text{\ensuremath{\max\{0,t_{1},\dots,t_{k-1}\}+\min\left\{ 0,t_{1},\dots,t_{k-1}\right\} }},0\right\} 
\]
with $\varphi_{\delta/10}\left(t\right)$, where $\varphi_{\varepsilon}\left(t\right)=\varepsilon^{-\left(k-1\right)}\varphi\left(t/\varepsilon\right)$,
and $\varphi\in C_{c}^{\infty}\left(\mathbb{R}^{k-1}\right)$ is the
standard mollifier. We then have
\[
R_{N_{m}}^{k}\left(L_{m},\alpha,\psi^{-}\right)\le R_{N_{m}}^{k}\left(L_{m},\alpha,\Delta\right)\le R_{N_{m}}^{k}\left(L_{m},\alpha,\psi^{+}\right);
\]
substituting the asymptotics (\ref{eq:Smooth_Subseq_R_k_asymp}),
we conclude that (\ref{eq:RkAlongSubseq}) holds for almost all $\alpha\in I$.
\end{proof}
We are now ready to prove Proposition \ref{prop:MainCLTProp}.
\begin{proof}[Proof of Proposition \ref{prop:MainCLTProp}]
 The argument is similar to that of the proof of Theorem \ref{thm:AlmostSureTheorem}.
Let $k\ge2$; it is enough to show that for almost all $\alpha\in I$,
we have
\[
R_{N}^{k}\left(L,\alpha,\Delta\right)=L^{k-1}+O\left(L^{-s}\right)
\]
for all $s>0$. Let $N_{m}=\lfloor m^{1+\epsilon/2}\rfloor$, so that
for any $N$ there exists $m$ such that $N_{m-1}\le N<N_{m}$. Moreover,
$\frac{N_{m}}{N}=1+O\left(m^{-1}\right)$, and by the assumption 
\[
L\left(N+1\right)-L\left(N\right)=O\left(N^{-\epsilon}\right)
\]
we have
\[
L=L_{m}+O\left(\frac{N-N_{m}}{N^{\epsilon}}\right)=L_{m}+O\left(N_{m}^{-\epsilon/2}\right);
\]
hence, there exists a constant $C>0$ such that for sufficiently large
$N$ we have 
\begin{align*}
R_{N}^{k}\left(L,\alpha,\Delta\right) & \le\frac{N_{m}}{N}R_{N_{m}}^{k}\left(L\cdot\frac{N_{m}}{N},\alpha,\Delta\right)\\
 & \le\left(1+Cm^{-1}\right)R_{N_{m}}^{k}\left(\left(L_{m}+CN_{m}^{-\epsilon/2}\right)\cdot\left(1+CN_{m}^{-\frac{1}{1+\epsilon/2}}\right),\alpha,\Delta\right);
\end{align*}
by the upper bound on $L$ and Lemma \ref{lem:k-corr_subsequence}
with $\left(L+CN^{-\epsilon/2}\right)\cdot\left(1+CN^{-\frac{1}{1+\epsilon/2}}\right)$
instead of $L$, for almost all $\alpha\in I$  we have
\begin{gather*}
R_{N_{m}}^{k}\left(\left(L_{m}+CN_{m}^{-\epsilon/2}\right)\cdot\left(1+CN_{m}^{-\frac{1}{1+\epsilon/2}}\right),\alpha,\Delta\right)\\
=\left(L_{m}+CN_{m}^{-\epsilon/2}\right)^{k-1}\cdot\left(1+CN_{m}^{-\frac{1}{1+\epsilon/2}}\right)^{k-1}+O\left(L_{m}^{-s}\right)=L_{m}^{k-1}+O\left(L_{m}^{-s}\right)
\end{gather*}
for all $s>0$. Thus, for sufficiently large $N$ we have (using again
the upper bound on $L$), for almost all $\alpha\in I$ we have
\begin{equation}
R_{N}^{k}\left(L,\alpha,\Delta\right)\le\left(1+Cm^{-1}\right)\left(L_{m}^{k-1}+O\left(L_{m}^{-s}\right)\right)=L_{m}^{k-1}+O\left(L_{m}^{-s}\right)=L^{k-1}+O\left(L^{-s}\right)\label{eq:k-corr_upper_bound}
\end{equation}
for all $s>0$. Similarly, for almost all $\alpha\in I$ we have
\begin{equation}
R_{N}^{k}\left(L,\alpha,\Delta\right)\ge L^{k-1}-O\left(L^{-s}\right)\label{eq:k-corr_lower_bound}
\end{equation}
for all $s>0$; the bounds (\ref{eq:k-corr_upper_bound}) and (\ref{eq:k-corr_lower_bound})
give (\ref{eq:corr_asymp}).
\end{proof}

\appendix

\section{Normal approximation to the Poisson distribution}

We require a normal approximation to a random variable whose moments
are close to the Poisson moments. Denote 
\[
\mu_{k}^{Poisson}\left(L\right)=\sum\limits _{j=0}^{k}\begin{Bmatrix}k\\
j
\end{Bmatrix}L^{j}
\]
the $k$-th moment of a Poisson-distributed random variable with parameter
$L$, and 
\[
\mu_{k}^{normal}=\begin{cases}
0 & k\,\text{odd}\\
\left(k-1\right)!! & k\,\text{even}
\end{cases}
\]
the $k$-th moment of a standard Gaussian random variable.
\begin{lem}
\label{lem:PoissonToNormal}Let $L=L\left(N\right)\to\infty$ as $N\to\infty$,
and let $\left(X_{N}\right)_{N=1}^{\infty}$ be a sequence of random
variables such that for all $j\ge1$ and for all $s>0$ we have 
\begin{equation}
\mathbb{E}\left[X_{N}^{j}\right]=\mu_{j}^{Poisson}\left(L\right)+O\left(L^{-s}\right)\label{eq:moments_condition}
\end{equation}
as $N\to\infty$. Then
\[
\frac{X_{N}-L}{\sqrt{L}}\stackrel{d}{\longrightarrow}\mathcal{N}\left(0,1\right)
\]
as $N\to\infty$, where $\mathcal{N}\left(0,1\right)$ is the standard
Gaussian distribution.
\end{lem}

\begin{proof}
Let $\widehat{X_{N}}=\frac{X_{N}-L}{\sqrt{L}}.$ It is sufficient
to prove that for all $k\ge1$ we have $\lim\limits _{N\to\infty}\mathbb{E}\left[\widehat{X_{N}}^{k}\right]=\mu_{k}^{normal}$.
By (\ref{eq:moments_condition}), we have
\[
\mathbb{E}\left[\widehat{X_{N}}^{k}\right]=L^{-k/2}\sum_{j=0}^{k}{k \choose j}\left(\mu_{j}^{Poisson}\left(L\right)\right)^{j}\left(-L\right)^{k-j}+o\left(1\right),
\]
so we have to show that for all $k\ge1$ we have
\begin{equation}
L^{-k/2}\sum_{j=0}^{k}{k \choose j}\left(\mu_{j}^{Poisson}\left(L\right)\right)^{j}\left(-L\right)^{k-j}=\mu_{k}^{normal}+o\left(1\right).\label{eq:normal_moments}
\end{equation}

Let $Y_{L}$ be a Poisson-distributed random variable with parameter
$L$ and $\widehat{Y_{L}}=\frac{Y_{L}-L}{\sqrt{L}}$; we have to show
that for all $k\ge1$ we have $\lim\limits _{N\to\infty}\mathbb{E}\left[\widehat{Y_{L}}^{k}\right]=\mu_{k}^{normal}$.
Let $M_{\widehat{Y_{L}}}\left(t\right)$ be the moment-generating
function of $\widehat{Y_{L}}$. Then for any $t$ we have
\begin{align}
M_{\widehat{Y_{L}}}\left(t\right) & =e^{-t\sqrt{L}+L\left(e^{t/\sqrt{L}}-1\right)}=e^{-t\sqrt{L}+L\left(t/\sqrt{L}+t^{2}/\left(2L\right)+O\left(L^{-3/2}\right)\right)}\nonumber \\
 & =e^{\frac{t^{2}}{2}+O(L^{-1/2})}\underset{N\to\infty}{\longrightarrow}e^{t^{2}/2}\label{eq:CharConv}
\end{align}
so that the limit is the moment-generating function of a standard
Gaussian random variable. Since the convergence in (\ref{eq:CharConv})
is uniform in a complex neighbourhood of $t=0$ and all the functions
involved are (complex) analytic, convergence of the moments (which
can be expressed as the derivatives of the moment-generating function
evaluated at zero) easily follows from Cauchy's integral formula.
\end{proof}

\end{document}